\documentclass{article}
\usepackage[all,ps,arc]{xy}  
\usepackage{amsmath,amssymb,latexsym}
\usepackage{amsfonts}
\usepackage{amsthm}
\newcommand{\point}{\ensuremath{\xymatrix{A\ar@<+.6ex>[r]^(.5){\alpha}
&B\ar@<+.6ex>[l]^(.5){\beta}}}}
\newcommand{\rg}{\ensuremath{\xymatrix{A\ar@<+1ex>[r]^{\alpha}\ar@<-1ex>[r]_{\gamma}&B\ar[l]|{\beta}}}}

\newtheorem{Theorem}{Theorem}[section]

\newtheorem{Proposition}[Theorem]{Proposition}

\newtheorem{Corollary}[Theorem]{Corollary}

\newtheorem{Remark}[Theorem]{Remark}
\newtheorem{Example}[Theorem]{Example}

\newcommand{\cC}{{\mathcal C}}

\newcommand{\cF}{{\mathcal F}}

\newcommand{\cK}{{\mathcal K}}
\newcommand{\cN}{{\mathcal N}}

\newcommand\Pt{\mathbf{Pt}}

\newcommand\K{\mathcal{K}}
\newcommand\coker{\mathrm{coker}}

\begin{document}

\def \tm{\!\times\!}

\newenvironment{changemargin}[2]{\begin{list}{}{
\setlength{\topsep}{0pt}
\setlength{\leftmargin}{0pt}
\setlength{\rightmargin}{0pt}
\setlength{\listparindent}{\parindent}
\setlength{\itemindent}{\parindent}
\setlength{\parsep}{0pt plus 1pt}
\addtolength{\leftmargin}{#1}\addtolength{\rightmargin}{#2}
}\item}{\end{list}}

\title{A note on strong protomodularity, \\
actions and quotients}
\author{Giuseppe Metere}

\maketitle

\begin{abstract}

  In order to study the problems of extending an action along a quotient of the acted
  object and along a quotient of the acting object,
  we investigate some properties of the fibration of points.
  In fact, we obtain a characterization of protomodular categories among quasi-pointed
  regular ones, and, in the semi-abelian case, a characterization of strong protomodular
  categories.
  Eventually, we return to the initial questions by  stating the results
  in terms of internal actions.
\end{abstract}

\section{Introduction}
The present work originates from the investigation of the categorical properties related
to two well-known features of group actions.

\medskip\noindent
{\bf Actions on quotients}

\noindent Suppose we are given  a pair $(\xi,g)$:
$$
\xymatrix{A\times Y\ar@{-->}[r]^{\xi}&Y\ar[r]^g&Z \,,}
$$
where $\xi$ is a left-action of groups, and $g$ is a surjective homomorphism. We discuss the following problem:
under what conditions does the action $\xi$ induces an action on the quotient $Z$?

Indeed, it is not difficult to see that  $\xi$ is well-defined on the cosets of $Y$ \emph{mod} $X=\mathrm{Ker}(g)$,
precisely when it is well-defined on the $0$-coset $X$, i.e.\ when it restricts to $X$.
We shall state this property as follows:
\begin{itemize}
\item[(KC)] An action passes to the quotient if, and only if, it restricts to the kernel.
\end{itemize}

\medskip\noindent
\noindent {\bf Action of quotients}

Suppose now that we are given a group action $\xi$ as before, and a surjective group homomorphism
$q\colon A\to Q$.
A natural question arises: when does the given $A$-action induce a $Q$-action?
In this case, the restriction of the action $\xi$ to the kernel $K$ of $q$ always exists, and the condition
under which the action of the quotient is well defined, amounts to the fact that the
kernel of $q$ acts trivially.

\medskip
\noindent
These issues can be addressed in any category where a notion of internal object action
is available, e.g.\ in any semi-abelian category (see \cite{JMT02}).
Indeed, we will show that the property (KC) characterizes strongly protomodular
categories among semi-abelian categories, and that, in such contexts,
actions of quotients behave substantially in the same way as in the case of groups.

On the other hand these issues can be dealt with also in more general contexts.
Indeed, when an object $A$ acts on object $X$, just like in the case of group, one
can consider the split epimorphism $X\rtimes A\to A$ given by the  semidirect product
projection together with its canonical section. Vice-versa, any split epimorphism
with codomain $A$ gives rise to the \emph{conjugation} $A$-action on the kernel of the split epimorphism.

This allows to formulate our issues in terms of split epimorphisms, or \emph{points},
even in contexts where the machinery of internal actions is not at all available.
This line of investigation will lead us to the study of some new classifying
aspects of the fibration of points. In particular, with Proposition \ref{prop:Ker_reflects_sec},
we will give a characterization of protomodular categories among quasi-pointed regular ones
as those with kernel functors that reflect short exact sequences. Then, we will show that the problem of
extending actions along quotients translates (in term of points) in a property
closely connected with strong protomodularity, i.e.\ the fact that kernel functors reflect kernels In fact,
this property coincides with strong protomodularity in the semi-abelian case (Proposition \ref{prop:strongly}). On the other hand, the property of extending
an action along a quotient of the acting object has a counterpart in terms of points
in a property of change of base functors, as  described in Proposition \ref{prop:trivial_pointed_pass_to_Q}.
This observation eventually provides an exhaustive description of change of base functors
of the fibration of points along a regular epimorphism.

Our work confirms that strongly protomodular categories are a convenient setting for
working with internal actions, and related constructions.
Indeed, in the (strongly semi-abelian) varietal case, not only internal actions can be
described externally, i.e.\ with suitable set-theoretical maps, but also, they behave
\emph{nicely} with respect to quotients. This fact allows to apply varietal
techniques to the intrinsic setting.

Many  varieties of  universal algebra are strongly protomodular:
the categories of groups, Lie algebras, rings and, more generally, all distributive
$\Omega_2$-groups, i.e.\  distributive $\Omega$-groups with only unary and binary
operations (see \cite{MM10b}), as for instance the categories  of interest in the
sense of G.\ Orzech \cite{Orz72}.

\medskip
The paper is organized as follows.

In the next section we recall the basic notions
and fix the notation.

The third and the fourth sections are quite independent to each other.

Section three is devoted to the study of the exactness properties of kernel functors.
We prove that in quasi-pointed regular categories,
protomodularity is equivalent to the fact that kernel functors reflect short exact
sequences. Then we give a characterization of strongly
semi-abelian categories among semi-abelian ones (Theorem \ref{thm}).

In the fourth section the context is assumed to be strongly semi-abelian.
Here we approach the problem of determining the conditions that make
it possible to factor the change of base functor  of the
fibration of points along a regular epimorphism as an equivalence of categories followed by
a full embedding.

Actions on quotients and actions of quotients are treated explicitly in section
five, where the results obtained in the previous sections
 are reconsidered in terms of internal object actions.

\section{Preliminaries}\label{sec:preliminaries}
Here we recall some basic notions from \cite{BB}, and fix the notation.
\subsection{Protomodularity}
Let $\cC$ be a category with finite limits. We denote by $\Pt(\cC)$ the category with objects the four-tuples
$(B,A,b,s_b)$ in $\cC$, with $b\colon B\to A$ and $b\cdot s_b=1_A$, and with morphisms
$(f,g)\colon (D,C,d,s_d)\to (B,A,b,s_b)$:
\begin{equation}\label{diag:split_epi_morphism}
\begin{aligned}\xymatrix{
D\ar@<-.5ex>[d]_{d}\ar[r]^{f}
&B\ar@<-.5ex>[d]_{b}
\\
C\ar@<-.5ex>[u]_{s_d}\ar[r]_g
&A\ar@<-.5ex>[u]_{s_b}}
\end{aligned}
\end{equation}
such that both the upward and the downward directed squares commute.
The codomain assignment $(B,A,b,s_b)\mapsto A$ gives rise to a fibration, the so called
\emph{fibration of points}:
$$
\cF\colon \Pt(\cC)\to\cC.
$$
For an object $A$ of $\cC$, we denote by $\Pt_A(\cC)$ the fiber of $\cF$ over $A$.
Cartesian morphism  are given by  commutative diagrams (\ref{diag:split_epi_morphism})
with the downward directed square a pullback.
This way, any morphism $g\colon C\to A$ defines a ``change of base'' functor
$g^*\colon \Pt_A(\cC)\to\Pt_C(\cC)$.

If the category $\cC$ is  finitely complete, also the fibers $\Pt_A(\cC)$ are, and every change of base
functor is left exact. In the present work, $\cC$ will be always  finitely complete.

A category $\cC$ is called \emph{protomodular} when every change of base
of the fibration of points is conservative, i.e.\ when it reflects  isomorphisms (see \cite{BB}).

When  $\cC$  admits an initial object $0$, for any object $A$ of $\cC$, one can consider
the change of base along the initial arrow $!_A\colon 0\to A$. This defines a \emph{kernel functor} $\K_A$, for every object $A$.
In the presence of an initial object, the protomodularity condition can be simplified
by requiring that just kernel functors are conservative.

The category $\cC$ is called quasi-pointed when the unique arrow $0\to 1$ is a monomorphism.
Considering this being the case, the domain functor $\Pt_0(\cC)\to \cC$  defines an embedding of categories.
Its isomorphic image is the subcategory $\cC_0$ spanned by objects with null support (i.e.\ objects $A$ equipped with a necessarily unique arrow $\omega_A\colon A\to 0$) so
that we  can factor
$$
\K_A\colon \Pt_A(\cC)\to\cC_0\hookrightarrow\cC\,.
$$

When $0\to 1$ is an isomorphism, we say that $\cC$ is pointed; if this is the case, clearly $\cC_0=\cC$.

\smallskip
Let $\cC$ be a quasi-pointed finitely complete category. We shall call
\emph{kernel map} any $f\colon X\to Y$, pullback of an initial arrow, i.e.\ when $f$ fits into a pullback diagram
as it is shown below:
\begin{equation}\label{diag:kernel}
\begin{aligned}
\xymatrix{
X\ar[r]^{\omega_X}\ar[d]_{f}
&0\ar[d]^{!_Z}
\\
Y\ar[r]_{g}
&Z
}
\end{aligned}
\end{equation}
In this case, we write $f=\mathrm{ker}(g)$ or $f=k_g$. We denote by $\mathcal K$ the class of
kernel maps of a  given category $\cC$.

Following \cite{Bo91}, we say that $g$ is the cokernel of $f$, and we write $g=\coker(f)$, when
$(\ref{diag:kernel})$ is a pushout. Let us notice that this definition of cokernel is not dual to that
of kernel given above, unless the category is pointed.
When both conditions above are satisfied, i.e.\ when $(\ref{diag:kernel})$ is at the same time both a pullback and a pushout,
we call the pair $(f,g)$ \emph{short exact sequence} (see \cite{Bo91}), and we describe it by the diagram:
$$
\xymatrix{X\ar[r]^f&Y\ar[r]^g&Z}\,
$$
We recall from \cite{Bo01} that, if $\cC$ is quasi-pointed and protomodular, every regular epimorphism is the cokernel of its kernel,
so that the pair $(f,g)$ is a short exact sequence precisely when $g$ is a
regular epimorphism, and $f$ is its kernel.

\smallskip
Recall that an (internal) equivalence relation is called \emph{effective} when it is
the kernel pair of a map. A category $\cC$ is \emph{regular}, if it is finitely
complete, it has pullback-stable regular epimorphisms, and all effective
equivalence relations admit coequalizers.
A regular category $\cC$ is \emph{Barr exact} when all equivalence relations are
effective (see \cite{Ba71}).

\smallskip
Quasi-pointed protomodular regular categories are called
\emph{sequentiable}. If they are in fact pointed, they are called \emph{homological},
and they are termed \emph{semi-abelian} when they are also Barr exact and with finite
coproducts (see \cite{BB}).

An important feature of sequentiable categories is that, in such contexts, intrinsic
versions of some classical lemmas of homological algebra hold.
This is the case of the $3\times 3$ lemma (see \cite{Bo01}), that will be a
basic tool in the development of the present work.

\subsection{Strong protomodularity}\label{subsec:strong_protomodularity}

In \cite{Bo00b}, Bourn introduces a more general notion of \emph{normal monomorphism} that objectifies an equivalence class of an internal equivalence relations.

In a category $\cC$ with finite limits, a morphism $f\colon X\to Y$ is \emph{normal}
to an equivalence relation $(R,r_1,r_2)$ on the object $X$ when the following
two diagrams are pullbacks:
$$
\xymatrix{\ar@{}[dr]|(.3){\lrcorner}
X\times X\ar[r]\ar@{=}[d]
&R\ar[d]^-{\langle r_1,r_2\rangle}
\\
X\times X \ar[r]_(.45){f\times f}
&Y \times Y
}
\qquad
\xymatrix{\ar@{}[dr]|(.3){\lrcorner}
X\times X\ar[d]_-{p_1}\ar[r]
&R\ar[d]^-{r_1}
\\
X\ar[r]_-{f}
&Y
}
$$
When the category $\cC$ is protomodular, normality becomes a property: if $f$ is
normal to a relation $R$, then $R$ is unique.
We denote by $\mathcal N$ the class of  normal monomorphisms.

Indeed, in quasi-pointed protomodular categories, any kernel is normal to its
associated kernel relation. On the other hand,
not every normal monomorphism is a kernel, i.e.\ $\mathcal{K}\subseteq \mathcal{N}$,
and the inclusion may be  strict, in general.

Let us recall from \cite{Bo00b} that if $\cC$ is finitely complete, pointed and protomodular,
then the class $\cK$ coincides with the class $\cN$ precisely when every equivalence
relation is effective.

\smallskip
In \cite{Bo01} Bourn calls \emph{normal}, a left exact functor that is conservative and
reflects normal monomorphisms. A relevant application of this definition is related to
the fibration of points. When all the change of base functors
are normal, the category is called  \emph{strongly-protomodular} (see \cite{Bo00, BB}).
In the presence of initial object, it suffices to consider  the kernel functors
$\K_A$, for every object $A$. A strongly protomodular semi-abelian category is termed
\emph{strongly semi-abelian}.

Bourn, in \cite{Bo00}, gives a characterization of normal subobjects in $\Pt_A(\cC)$.
When $\cC$ is quasi-pointed protomodular, $\varphi\colon (B,b,s_b)\to (C,c,s_c)$
is normal in $\Pt_A(\cC)$ if, and only if, $\varphi\cdot k_b$ is normal in $\cC$, where
$X=\mathrm{Ker}(b)$ and $Y=\mathrm{Ker}(c)$:
$$
\xymatrix{
X\ar[r]^{k_b}\ar[d]_{f}
&B
\ar@<+.5ex>[r]^{b}\ar[d]^{\varphi}
&A\ar@<+.5ex>[l]^{s_b}\ar@{=}[d]
\\
Y\ar[r]_{k_c}
&C
\ar@<+.5ex>[r]^{c}
&A\ar@<+.5ex>[l]^{s_c}
}
$$
This, in turns, gives a criterion for strong protomodularity: it suffices to check, for every
morphism of split short exact sequences as above, that if $f$ is normal, then also $k_c\cdot f$ is.

\section{Exactness properties of kernel functors}\label{sec:exactness properties of kernel functors}
In this section, we analyze some issues related to the behavior of kernel functors with
respect to kernels, cokernels and short exact sequences, in the quasi-pointed regular
setting. As recalled before, in this case, the kernel functor takes values in the base category $\cC$. Moreover, when
$\cC$ is protomodular, regular (or Barr-exact), then also $\Pt_A(\cC)$ is protomodular, regular (or Barr-exact)
respectively (see   \cite{BB}, for instance). These circumstances suggest to investigate
the exactness properties of kernel functors, in the sense of  homological algebra.

Our main motivation rests in the observation that a notion similar to strong protomodularity, but stated in terms of kernels instead of normal monomorphisms, is connected with (actually equivalent to) the problem of extending actions along quotients. This connection will be made explicit in the next sections.

\smallskip

The preservation property described in the next proposition is little  more than a reformulation of
some arguments analyzed in \cite{Bo01}.
\begin{Proposition}\label{prop:Ker_preserves_sec}
Let $\cC$ be a quasi-pointed protomodular category with pullback stable regular epimorphisms.
Then  $\K_A$ preserves short exact sequences, for every $A$ in $\cC$.
\end{Proposition}
\begin{proof}
Let us consider a short exact sequence $(\varphi,\gamma)$, and let
$f=\K_A(\varphi)$ and $g=\K_A(\gamma)$, as described by the following diagram.
\begin{equation}\label{diag:K*C*}
\begin{aligned}
\xymatrix{
\K_A(b)\ar[r]^{k_b}\ar[d]_{f}
&B
\ar@<+.5ex>[r]^{b}\ar[d]_{\varphi}
&A\ar@<+.5ex>[l]^{s_b}\ar@{=}[d]
\\
\K_A(c)\ar[r]_{k_c}\ar[d]_{g}
&C
\ar@<+.5ex>[r]^{c}\ar[d]_{\gamma}
&A\ar@<+.5ex>[l]^{s_c}\ar@{=}[d]
\\
\K_A(d)\ar[r]_{k_d}
&D\ar@<+.5ex>[r]^{d}
&A\ar@<+.5ex>[l]^{s_d}
}
\end{aligned}
\end{equation}
Since kernel functors preserve limits, $\ker(g)=f$. Furthermore, by Lemma 1 in \cite{Bo01},
the left-down square is a pullback, and since $\gamma$ is a regular epimorphism, so is $g$.
Finally, by Proposition 2 in  \cite{Bo01}, $g=\coker(f)$.
\end{proof}

\smallskip
Before we can treat  reflection properties of kernel functors, let us develop
the necessary description of kernels in $\Pt_A(\cC)$.

We have just recalled Bourn's characterization of normal subobject in $\Pt_A(\cC)$.
In the sequentiable setting, one can recover a similar characterization for kernels.

\begin{Proposition}\label{prop:characterization_of_kernels}
In a sequentiable category $\cC$, let us consider a morphism of points
$\varphi\colon (B,b,s_b)\to (C,c,s_c)$, together with its restriction to kernels, as described
by the commutative diagram below:
\begin{equation}\label{diag:phi}
\begin{aligned}
\xymatrix{
X\ar[r]^{k_b}\ar[d]_{f}
&B
\ar@<+.5ex>[r]^{b}\ar[d]^{\varphi}
&A\ar@<+.5ex>[l]^{s_b}\ar@{=}[d]
\\
Y\ar[r]_{k_c}
&C
\ar@<+.5ex>[r]^{c}
&A\ar@<+.5ex>[l]^{s_c}
}
\end{aligned}
\end{equation}
Then
\begin{itemize}
\item[$(1)$] $\varphi$ is a kernel in $\Pt_A(\cC)$ if, and only if,
$\varphi\cdot k_b$ is a kernel in $\cC$;
\item[$(2)$] in this case, the cokernel of $\varphi$ in $\Pt_A(\cC)$ is given by
the cokernel of $\varphi\cdot k_b$ in $\cC$.
\end{itemize}
\end{Proposition}
\begin{proof}
Point $(1)$.
The fact that $\varphi$ is a kernel, amounts to the existence of a morphism of points
$\gamma\colon (C,c,s_c)\to (D,d,s_d) $, such that the commutative  square $\gamma\cdot \varphi=s_d\cdot b$
is a pullback in $\cC$. Then, pasting it with the kernel diagram of $(k_b,b)$, one easily sees that $k_c\cdot f=\varphi\cdot k_b=\mathrm{ker}(\gamma)$ in $\cC$.

Conversely, let us assume that $k_c\cdot f=\varphi\cdot k_b$ is a kernel in $\cC$,
and let $\gamma\colon C\to D$ be its cokernel (always in $\cC$). Then $\gamma$ underlies a morphism of points.
Indeed, since $c\cdot\varphi\cdot k_b$ factors through $0$, we get a unique
$d\colon D\to A$ such that $d\cdot \gamma = c$. In fact, $d$ is a split epimorphism
with section $s_d=\gamma\cdot s_c$, and $\gamma\colon (C,c,s_c)\to (D,d,s_d)$ is a
morphism of points. We are to prove that $\varphi$ is the kernel of $\gamma$ in $\Pt_A(\cC)$.
To this end, let us consider the commutative diagram
$$
\xymatrix{
X\ar[r]^{k_b}\ar[d]
&B\ar[r]^{\varphi}\ar@<+.5ex>^{b}[d]
&C\ar[d]^{\gamma}
\\
0\ar[r]
&A\ar@<+.5ex>[u]^{s_b}\ar[r]_{s_d}
&D}
$$
The whole diagram and the square on the left are pullbacks, so that by  the pullback
cancelation property of protomodular categories (see \cite{Bo91}) also the square on
the right is a pullback, thus showing that $\varphi$ is the kernel of $\gamma$ in
$\Pt_A(\cC)$.

\smallskip
Point $(2)$.
Clearly $\Pt_A(\cC)$ is pointed, moreover it is protomodular and regular (see \cite{BB}) as $\cC$ is.
Actually, as showed in the proof of point 1, $\gamma=\coker(\varphi\cdot k_c)$ in $\cC$ underlies
a regular epimorphism. In order to conclude the proof, it suffices to recall that in
 homological categories a regular epimorphism  is always the cokernel of its kernel.
\end{proof}

\smallskip
We are now ready to show how, in the sequentiable setting, kernel functors also reflect short exact sequences.
Moreover, this property characterizes sequentiable categories among quasi-pointed
regular ones.

\begin{Proposition}\label{prop:Ker_reflects_sec}
Let $\cC$ be a quasi-pointed regular category. The following statements are equivalent:
\begin{itemize}
\item[$(1)$] $\cC$ is protomodular,
\item[$(2)$] $\K_A$ reflects short exact sequences, for every $A$ in $\cC$.
\end{itemize}
\end{Proposition}
\begin{proof}
In a protomodular category $\cC$, let us consider a pair $(\varphi,\gamma)$ of morphisms of points over $A$, such that
applying the kernel functor $\cK_A$ one obtains a short exact sequence $(f,g)$, see diagram (\ref{diag:K*C*}).
Since $g\cdot f=0$, $\gamma\cdot \varphi$ factors through $A$. More precisely,
 $\gamma\cdot \varphi=s_d\cdot b$, as  one can prove by pre-composing this equality with the jointly  epic pair $(k_b,s_b)$.
Then we consider the diagram below:
\begin{equation}\label{diag:3x3}
\begin{aligned}
\xymatrix{
\K_A(b)\ar@{=}[r]\ar[d]_{f}
&\K_A(b)
\ar[r]\ar[d]^{\varphi\cdot k_b}
&0\ar[d]
\\
\K_A(c)\ar[r]_{k_c}\ar[d]_{g}
&C
\ar@<+.5ex>[r]^{c}\ar[d]_{\gamma}
&A\ar@<+.5ex>[l]^{s_c}\ar@{=}[d]
\\
\K_A(d)\ar[r]_{k_d}
&D\ar@<+.5ex>[r]^{d}
&A\ar@<+.5ex>[l]^{s_d}
}
\end{aligned}
\end{equation}
We can apply the $3\times 3$ lemma: the three rows are short exact, and so are the
leftmost and the rightmost columns. The middle column is zero, hence we can conclude
that it is short exact. By Proposition \ref{prop:characterization_of_kernels},
the pair $(\varphi,\gamma)$ is a short exact sequence in $\Pt_A(\cC)$.

Conversely, we have to prove that, for any object $A$, the kernel functor $\K_A$ reflects isomorphisms.
To this end, we consider a map $\varphi$ in $\Pt_A(\cC)$ such that its restriction to kernels
is an isomorphism $f$. Then, since kernels have null support, the cokernel of $f$ exists, and
of course it is trivial. Thus one can consider the following diagram:
$$
\xymatrix{
\K_A(b)\ar[r]^{k_b}\ar[d]_{f}^{\simeq}
&B
\ar@<+.5ex>[r]^{b}\ar[d]_{\varphi}
&A\ar@<+.5ex>[l]^{s_b}\ar@{=}[d]
\\
\K_A(c)\ar[r]_{k_c}\ar[d]
&C
\ar@<+.5ex>[r]^{c}\ar[d]_{c}
&A\ar@<+.5ex>[l]^{s_c}\ar@{=}[d]
\\
0\ar[r]
&A\ar@<+.5ex>[r]^{1_A}
&A\ar@<+.5ex>[l]^{1_A}
}
$$
By applying the hypothesis, we obtain that the sequence $(\varphi,c)$ is short exact in $\Pt_A(\cC)$, so that $\varphi$ is the pullback of  $1_A$ along $c$, hence an isomorphism.

\end{proof}
Proposition \ref{prop:Ker_preserves_sec} and Proposition \ref{prop:Ker_reflects_sec} together, imply immediately
the following corollary.
\begin{Corollary}
Let $\cC$ be sequentiable. Then for any map $e\colon E\to A$, the change of base
$e^*\colon\Pt_A(\cC)\to \Pt_E(\cC)$ preserves and reflects short exact sequences.
\end{Corollary}

\medskip
In the last part of  this section we would like to examine the behavior of the kernel functors with respect to kernels
and (some specific class of) cokernels.
We start by considering a distinguished class of morphisms of points, i.e.\ those maps
$\varphi$ in $\Pt_A(\cC)$ such that their restriction to the kernel functor $\K_A$
is a kernel map in $\cC$.

Of course, if  $\varphi$ is a kernel in $\Pt_A(\cC)$, then $\cK_A(\varphi)$ is
a kernel in $\cC$. On the other hand, we wish to investigate when the other implication holds.
This is done in the next proposition.

\begin{Proposition}\label{prop:proto_strongly}
Let $\cC$ be sequentiable, and let $\K_A\colon\Pt_A(\cC)\to\cC$ be the kernel functor relative to an object $A$ in $\cC$.
Then the following statements are equivalent:
\begin{itemize}
\item[$(1)$] the kernel functor $\K_A$ reflects kernel maps,
\item[$(2)$] the kernel functor $\K_A$ \emph{lifts} the cokernels of the maps $\varphi$
  such that $\K_A(\varphi)$ is a kernel, i.e.:
\item[$(\mathrm{C*})$] for every $\varphi$ in $\Pt_A(\cC)$ such that $\K_A(\varphi)=f$ is a kernel map in $\cC$, there exists a unique $\gamma$ in $\Pt_A(\cC)$
such that $\gamma=\mathrm{coker}(\varphi)$, and  $\K_A(\gamma)=\mathrm{coker}(f)$.
\end{itemize}
\end{Proposition}
For the notion of (co)limit lifting functor, the reader can refer to Definition 13.17 of \cite{Ad90}.
\begin{proof}
$(1)\Rightarrow (2)$. Let us consider a morphism $\varphi\colon (B,b,s_b)\to (C,c,s_c)$ such that
$f=\cK_A(\varphi)$ is a kernel, and let $g=\coker(f)$. By $(1)$ $\varphi$ is a kernel in $\Pt_A(\cC)$,
hence by Proposition \ref{prop:characterization_of_kernels}, also $ k_c\cdot f$ is.

Let $(D,\gamma)$ be the cokernel of $k_c\cdot f$, so that the pair $(k_c\cdot f,\gamma)$ is
a short exact sequence. We can consider the following commutative diagram
$$
\xymatrix{
\cK_A(b)\ar@{=}[r]\ar[d]_{f}
&\cK_A(b)\ar[r]\ar[d]^{ k_c\cdot f}
&0\ar[d]
\\
\cK_A(c)\ar[r]^{k_c}\ar[d]_{g}
&C\ar@<+.5ex>[r]^c\ar[d]^{\gamma}
&A\ar@<+.5ex>[l]^{s_c}\ar@{=}[d]
\\
Z\ar[r]_{\alpha}
&D\ar[r]_{\beta}
&A
}
$$
where $\alpha$ and $\beta$ are obtained by the universal properties of the cokernels
$(D,\gamma)$ and $(Z,g)$.
Since $\cC$ is sequentiable, we can apply the $3\times 3$ lemma, and  conclude
that the sequence $(\alpha,\beta)$ is short exact. Moreover, $\beta$ is split by
$\gamma\cdot s_c$, so that $\gamma$ induces a morphism of points
$(C,c,s_c)\to(E,\beta, \gamma\cdot s_c)$; in fact, by Proposition \ref{prop:characterization_of_kernels},
$\gamma =\coker(\varphi)$ in $\Pt_A(\cC)$.

Now, universality of kernels implies the existence of a unique isomorphism
$\tau\colon Z\to \cK_A(\beta)$ such that $\alpha=k_{\beta}\cdot \tau$.
Of course, as $g$ is a cokernel of $f$ also $\tau\cdot g$ is, moreover
$\cK_A(\gamma)=\tau\cdot g$, so that the existence part of $(\mathrm{C*})$ is
granted. Uniqueness comes from the fact that, since $\cC$ is protomodular,
the pair $(k_c,s_c)$ is jointly strongly epic.

$(2)\Rightarrow (1)$. Let us assume that  $f=\cK_A(\varphi)$ is a kernel.
By $(C*)$ then, there is a $\gamma=\coker(\varphi)$ in $\Pt_A(\cC)$ such that
$\cK_A(\gamma)=\coker f$. Now apply Proposition \ref{prop:Ker_reflects_sec} and get
$(\varphi,\gamma)$ is short exact. In particular, $\varphi$ is a kernel.
\end{proof}

Whenever kernel maps and normal monomorphisms coincide, condition $(1)$
of Proposition \ref{prop:proto_strongly} above, expresses precisely the strong
protomodularity axiom. This proves the following proposition.

\begin{Proposition}\label{prop:strongly}\label{prop:C*}
Let $\cC$ be a semi-abelian category. The following statements are equivalent
\begin{itemize}
\item[$(1)$] $\cC$ is strongly semi-abelian,
\item[$(2)$] for every object $A$ of $\cC$, the kernel functor $\K_A$ lifts the cokernels of $\K_A$-kernels, i.e.\ condition
$(C*)$ of Proposition \ref{prop:proto_strongly} is satisfied.
\end{itemize}
\end{Proposition}

\section{Change of base: the other direction}\label{sec:pf_change_of_base}

As we have recalled in Section \ref{sec:preliminaries},  for any map $f\colon E\to A$,
the change of base functor $f^*\colon \Pt_A(\cC) \to \Pt_E(\cC)$, is defined by pulling
back along $f$. In other terms, $f$ defines a functor between the fibers that moves
\emph{backward}, with respect to the direction of $f$.
A quite natural question to ask is whether there  are conditions allowing to
\emph{push forward} along a map. More precisely, given a map $q\colon A\to Q$, we aim
to define a functor $q_*\colon  \Pt_A(\cC) \to \Pt_Q(\cC)$.

In the present work, we restrict our attention to the case when $q$ is a regular
epimorphism. The following result shows that such a push forward can be performed
if, and only if, the pullback along $k=\ker(q)$ trivializes the pointed object we started with:

\begin{Proposition}\label{prop:trivial_pointed_pass_to_Q}
In a strongly semi-abelian category $\cC$, we consider a pointed object $(C,c,s_c, A)$, and a
regular epimorphism $q\colon A\to Q$. Then, if we denote by $(K,k)$ and $(Y,k_c)$ the kernels of $q$ and of $c$
respectively, the following statements are equivalent:
\begin{itemize}
\item[$(1)$] the pullback along $k$ of   $(C,c,s_c, A)$ is the pointed object
$$
(Y\times K, \pi_2, \langle0,1\rangle, K)\,,
$$
\item[$(2)$] there exist a pointed object  $(D,d,s_d, Q)$  and a cartesian morphism
$$
(\gamma,q)\colon (C,c,s_c, A)\to (D,d,s_d, Q)\,.
$$
\end{itemize}
\end{Proposition}
The situation is described by the following diagram.
\begin{equation}\label{diag:q_*}
\begin{aligned}
\xymatrix{
Y\ar[r]^(.4){k_{\pi_2}}\ar@{=}[d]
&Y\times K\ar@{}[dr]|{(\lozenge)}
\ar@<+.5ex>[r]^(.6){\pi_2}\ar[d]_{\varphi}
&K\ar@<+.5ex>[l]^(.4){\langle0,1\rangle}\ar[d]^{k}
\\
Y\ar[r]_{k_c}\ar@{=}[d]
&C\ar@{}[dr]|{(\lozenge\lozenge)}
\ar@<+.5ex>[r]^{c}\ar@{-->}[d]_{\gamma}
&A\ar@<+.5ex>[l]^{s_c}\ar[d]^q
\\
Y\ar@{-->}[r]_{k_d}
&D\ar@{-->}@<+.5ex>[r]^{d}
&Q\ar@{-->}@<+.5ex>[l]^{s_d}
}
\end{aligned}
\end{equation}

\begin{proof}
$(1)\Rightarrow(2)$. By the assumption in $(1)$,  $\varphi$ is a kernel, since it is the pullback of a kernel.
Now, if we focus on the square $(\lozenge)$ above, we can consider the kernels of the horizontal
split epimorphisms and the cokernels of the vertical monomorphisms. Since the base category is homological,
not only can we say that such  kernels are isomorphic, but also the cokernels of $k$ and $\varphi$ are. This last
claim is proved by applying the $3\times 3$ lemma to the diagram
$$
\xymatrix{
Y\ar@{=}[d]\ar[r]^-{\langle1,0\rangle}
&Y\times K
\ar@<+.5ex>[r]^-{\pi_2}\ar[d]_{\varphi}
&K\ar@<+.5ex>[l]^{\langle0,1\rangle}\ar[d]^k
\\
Y\ar[d]\ar[r]^{k_c}
&
C\ar[d]_-{q\cdot c}
\ar@<+.5ex>[r]^{c}
&A\ar@<+.5ex>[l]^{s_c}\ar[d]^{q}
\\
0\ar[r]
&Q\ar@{=}[r]
&Q}
$$
Now, let us consider the following morphism of short exact sequences:
$$
\xymatrix{
K\ar[r]^k\ar[d]_{\langle0,1\rangle}
&A\ar[r]^q\ar[d]^{s_c}
&Q\ar@{=}[d]
\\
Y\times K\ar[r]_{\varphi}
&C \ar[r]_{q\cdot c}
&Q
}
$$
Since $\langle0,1\rangle$ is a kernel, applying Axiom M1.2 of \cite{Ro04} (which holds in every strongly semi-abelian
category) we deduce that also $s_c\cdot k=\varphi\cdot\langle0,1\rangle$ is a kernel. Let us compute the cokernel
 $\gamma=\mathrm{coker}(s_c\cdot k)$, and arrange our data in the diagram below:
$$
\xymatrix{
K\ar[r]^-{\langle0,1\rangle}\ar@{=}[d]
&Y\times K \ar[d]_{\varphi}\ar[r]^-{\pi_1}
&Y\ar[d]^{\alpha}
\\
K\ar[d]\ar[r]_{s_c \cdot k}
&C\ar[r]_{\gamma}\ar[d]^{q\cdot c}
&D\ar[d]^{\beta}
\\
0\ar[r]
&Q\ar@{=}[r]
&Q
}
$$
where $\alpha$ and $\beta$ are  obtained by the universal property of the cokernels
involved.
By the $3\times 3$ lemma, we deduce that the sequence $(\alpha,\beta)$  is short exact.
Finally, the square $\beta \cdot \gamma = q\cdot c$ is a pullback, since
$\mathrm{Ker}(\beta)=\mathrm{Ker}(c)$:
$$
\xymatrix{
&K\ar[d]_{s_c\cdot k}\ar@{=}[r]
&K\ar[d]^k
\\
Y\ar[r]^{k_c}\ar@{=}[d]
&C
\ar@<+.5ex>[r]^{c}\ar[d]^{\gamma}
&A\ar@<+.5ex>[l]^{s_c}\ar[d]^{q}
\\
Y\ar[r]_{\alpha}
&D
\ar[r]_{\beta}
&Q
}
$$
Then also $\mathrm{Ker}(\gamma)=\mathrm{Ker}(q)$.
Moreover $\beta$ is a split epimorphism. In order to prove this assertion, we notice that
 $(\gamma\cdot s_c)\cdot k=\gamma\cdot (s_c\cdot k)=0$, and
 by the universal property of the cokernel $q$, there exists a (unique) map $\sigma\colon Q\to D$
such that $\sigma \cdot q= \gamma\cdot s_c$.
Hence $\beta\cdot\sigma \cdot q=\beta \cdot \gamma\cdot s_c= q \cdot c\cdot s_c=q$, and, since $q$ is epic, we
get $\beta\cdot\sigma =1_Q$.

$(2)\Rightarrow(1)$. Assume we are in the situation as described by the diagram below
$$
\xymatrix{
B\ar@{}[dr]|{(i)}
\ar@<+.5ex>[r]^{b}\ar[d]_{\varphi}
&K\ar@<+.5ex>[l]^{s_b}\ar[d]^{k}
\\
C\ar@{}[dr]|{(ii)}
\ar@<+.5ex>[r]^{c}\ar[d]_{\gamma}
&A\ar@<+.5ex>[l]^{s_c}\ar[d]^q
\\
D\ar@<+.5ex>[r]^{d}
&Q\ar@<+.5ex>[l]^{s_d}
}
$$
with $(i)$ and $(ii)$ pullbacks, and $(k,q)$ short exact. Let us denote by $Y$ the kernel of $c$. Then $(i) + (ii)$ is a pullback, and since $q\cdot k $ factors
through $0$, the pointed object $(B,b,s_b,K)$ is isomorphic to the product projection $(Y\times K,\pi_2,\langle0,1\rangle,K)$.
\end{proof}

For a map $k$ with codomain $A$, we denote by $\Pt_A(\cC)|_k$ the full subcategory of $\Pt_A(\cC)$, with objects
 those split epimorphisms such that the change of base along $k$ gives a product projection.

Then it is easy to prove that Proposition \ref{prop:trivial_pointed_pass_to_Q} above can be used in order to
 give a description of the change of base $q^*$ when $q$ is a regular epimorphism.
\begin{Corollary}
Given a regular epimorphism $q\colon A\to Q$ together with its kernel $k=\ker(q)$ in a strongly semi-abelian category $\cC$,  we have a factorization $q^*=j\cdot e$
$$
\xymatrix{\Pt_Q(\cC)\ar@/^4ex/[rr]^{p^*} \ar[r]_{e}&\Pt_A(\cC)|_k  \ar[r]_{j} &\Pt_A(\cC)}
$$
where the functor $e$ is an equivalence of categories.
\end{Corollary}
\begin{proof}
For $\cC$ sequentiable, it is trivial to show that the change of base along a regular epimorphism $q$ is fully faithful.
Moreover, in the strongly semi-abelian case,point (2) of  Proposition \ref{prop:trivial_pointed_pass_to_Q} defines precisely a quasi-inverse
for the equivalence
$$
\Pt_Q(\cC)\to \Pt_A(\cC)|_k\,.
$$
\end{proof}

\section{Back to action(s)}\label{sec:back to actions}

In this section we  return to the problems described in the introduction, now set in
the semi-abelian context.

\subsection{Internal actions}\label{subsec:internal actions}
Semi-abelian categories are a convenient setting for working with
\emph{internal actions}. Here we briefly recall their definition from \cite{BJK}.
This will help in formulating internally the property (KC) of the introduction.

Let $\cC$ be a finitely complete, pointed category with  pushouts of split monomorphisms.
Then, for every object $A$ of $\cC$, the functor $\K_A$ has a left adjoint $\Sigma_A$. This
can be described as follows: for an object $X$ of $\cC$,
$\Sigma_A(X)$ is the pointed object $\xymatrix{A+X\ar@<+.5ex>[r]^(.6){[1,0]}&A\ar@<+.5ex>[l]^(.4){i_A}}$.
The monad corresponding to this adjunction is denoted by
$A\flat(-)$, and for any object $X$ of $\cC$ one gets a kernel diagram:
$$
\xymatrix{A\flat X\ar[r]^{\kappa_{A,X}}&A+X\ar[r]^(.6){[1,0]}&A}\,.
$$
The $A\flat(-)$-algebras are called internal $A$-actions (see \cite{BJK,BJ98}). The category  $\mathbf{Alg}(A\flat (-))$
 of such algebras
will be more conveniently denoted by $\mathbf{Act}(A,-)$.

When the kernel functor $\K_A$ is monadic, then $\cC$ is said to be a category with semi-direct products, and
the canonical comparison
\begin{equation}\label{eq:Pr=Act}
\Xi\colon\Pt_A(\cC) \to \mathbf{Act}(A,-)\,,
\end{equation}
 establishes an equivalence of categories. All semi-abelian categories satisfy this condition.

\begin{Example}\label{ex:group_actions}{\em
In the case of group one easily recovers the classical notion of group action.  For two given
groups $A$ and $X$, $A\flat X$ is noting but the subgroup of the free product $A*X$ generated by the words
$a;x;-a$, with $a\in A$ and $x\in X$, and the homomorphism $\xi\colon A\flat X\to X$ recovers a classical group action by
letting $a\cdot x= \xi(a;x;-a)$.
}\end{Example}

For an action $\xi\colon A\flat X\to X$, the semi-direct product of $X$ with $A$, with action $\xi$ is
 the split epimorphism corresponding to $\xi$ via $\Xi$.
It can be computed explicitly (see \cite{MM10b}) by means of the coequalizer diagram:
$$
\xymatrix{A\flat X\ar@<+.5ex>[r]^{\kappa_{A,X}}\ar@<-.5ex>[r]_{i_X\cdot\xi}&A+X\ar[r]^{q_\xi}&A\rtimes_{\xi}X\,.}
$$

\begin{Example}
{\em For objects $A$ and $X$, the \emph{trivial action} of $A$ on $X$ is  the composite
$$
\rho_{A,X}=\rho_X\colon\xymatrix{A\flat X\ar[r]^{\kappa_{A,X}}&A+X\ar[r]^(.6){[0,1]}&X}.
$$
The map $\rho_{A,X}$ is natural in the two variables $A$ and $X$.
The corresponding split epimorphism is given by the cartesian product with the canonical section:
 $$\xymatrix{X\times A\ar@<+.5ex>[r]^(.65){\pi_2}&A\ar@<+.5ex>[l]^(.35){\langle0,1\rangle}}\,.$$
}\end{Example}
\begin{Example}
{\em Every object  $X$ acts on itself by conjugation.  This is given by the composite
$$
\chi_X\colon\xymatrix{X\flat X\ar[r]^{\kappa_{X,X}}&X+X\ar[r]^(.6){[1,1]}&X}.
$$
The map $\chi_X$ is natural in the variable $X$.
The corresponding split epimorphism is isomorphic to the cartesian product with the diagonal section:
 $$\xymatrix{X\times X\ar@<+.5ex>[r]^(.65){\pi_2}&X\ar@<+.5ex>[l]^(.35){\langle1,1\rangle}}\,.$$
}\end{Example}

\subsection{Property (KC) for split epimorphisms}\label{subsec:property KC for split epimorphisms}
From now on, we consider $\cC$ semi-abelian.
In  this setting, we will first  formulate our property (KC) in terms of split epimorphisms,
according to the equivalence (\ref{eq:Pr=Act}) between actions and points.
Then, in the next section, we will go back to the original formulation of the problem.

Let a short exact sequence $\xymatrix{X\ar[r]^f&Y\ar[r]^g&Z}$ be given, and consider a split  epimorphism $(C,c,s_c)$,
  with kernel $Y$ and codomain $A$. Let $\xi$ be the corresponding action.
\begin{equation}\label{diag:K*C*2}
\begin{aligned}
\xymatrix{
X\ar@{-->}[r]^{k_b}\ar[d]_{f}
&B\ar@{}[dr]|{(\dag)}
\ar@{-->}@<+.5ex>[r]^{b}\ar@{-->}[d]_{\varphi}
&A\ar@{-->}@<+.5ex>[l]^{s_b}\ar@{==}[d]
\\
Y\ar[r]_{k_c}\ar[d]_{g}
&C\ar@{}[dr]|{(\ddag)}
\ar@<+.5ex>[r]^{c}\ar@{-->}[d]_{\gamma}
&A\ar@<+.5ex>[l]^{s_c}\ar@{==}[d]
\\
Z\ar@{-->}[r]_{k_d}
&D\ar@{-->}@<+.5ex>[r]^{d}
&A\ar@{-->}@<+.5ex>[l]^{s_d}
}
\end{aligned}
\end{equation}
With reference to the diagram above, the fact that the action $\xi$ restricts to the kernel $X$, amounts to the fact that there
exists a morphism $\varphi$ of split epimorphisms $(\dag)$, that restricts to $f$, while
the fact that the action $\xi$ passes to the quotient $ Z=Y/X$  amounts to the fact that there
exists a morphism $\gamma$ of split epimorphisms $(\ddag)$, that restricts to $g$.
In this fashion, with a little abuse of language, we can translate property (KC) in the double implication
$(\dag)\Leftrightarrow(\ddag)$.

\begin{Remark}{\em
Indeed, the implication $(\ddag)\Rightarrow(\dag)$ holds in any pointed category with finite limits, with no assumption on $g$.
In other words, this is the trivial part of our  problem, and it has nothing to do with actions, etc.

 On the other hand, the implication $(\dag)\Rightarrow(\ddag)$ translates precisely the condition $(C*)$ stated in
Proposition \ref{prop:proto_strongly}.
}\end{Remark}

\subsection{Actions on quotients}\label{subsec:action_on_quotients}
We are ready to return to our initial problem, and to formulate it in terms of internal actions.

\medskip
In a pointed regular category with semi-direct products, we consider
an  $A$-action $\xi$ on $Y$  and a short exact sequence, $(f,g)$.
\begin{equation}\label{diag:KC}
\begin{aligned}
\xymatrix{
A\flat X\ar[r]^{1\flat f}\ar@{-->}[d]_{\xi_{|}}\ar@{}[dr]|{(\dag)}
&A\flat Y\ar[r]^{1\flat g}\ar[d]_{\xi}\ar@{}[dr]|{(\ddag)}
&A\flat Z\ar@{-->}[d]_{\overline{\xi}}
\\
X\ar[r]_{f}
&Y\ar[r]_{g}
&Z
}
\end{aligned}
\end{equation}
We can state the implications above using internal actions, as follows:

\medskip

\begin{description}
\item[$(\ddag)\Rightarrow(\dag)$]
If $\xi$ induces an action on the quotient $Z$, then it restricts to the kernel $X$. In other words, if there exists an action $\overline \xi$ such that the square on the right commutes,
then there exists an action $\xi_{|}$ such that the square on the left commutes.
\item[$(\dag)\Rightarrow(\ddag)$]
If $\xi$ restricts to the kernel $X$,  then it induces an action on the quotient $Z$. In other words, if there exists an action $\xi_{|}$ such that the square on the left commutes,
then there exists an action $\overline\xi$ such that the square on the right commutes.
\end{description}

Of course, the implication $(\ddag)\Rightarrow(\dag)$ does hold in any pointed category with semi-direct products.
For what concerns property $(\dag)\Rightarrow(\ddag)$, we can translate Proposition \ref{prop:C*} accordingly,
in terms of internal actions. This is summarized  in the following Theorem, that can be derived directly from
Proposition \ref{prop:C*}.
\begin{Theorem}\label{thm}
Let $\cC$ be a semi-abelian category. The following statements are equivalent:
\begin{itemize}
\item[$(1)$] $\cC$ is strongly protomodular,
\item[$(2)$]  $(\dag)\Rightarrow(\ddag)$,
i.e.\ for any $A$-action $\xi$ on an object $Y$, and for any normal subobject $X$ of $Y$ such that
$\xi$ restricts $X$,   $\xi$ induces an action on the quotient $Y/X$.
\end{itemize}
\end{Theorem}

\subsection{Action of quotients}\label{subsec:action_of_quotients}
So far we discussed the conditions under which an action on a given object extends to a quotient of that object. Now
we change our point of view: we fix the \emph{acted} object, and we consider when an action \emph{of} a given object,
induces an action of a quotient of that object.

More precisely, let a short exact sequence
$$
\xymatrix{K\ar[r]^{k}&A\ar[r]^{q}&Q}
$$
be given, and let us consider an action $\xi\colon A\flat Y\to Y$. We pose the following question: when does the action
$\xi$ induce an action $Q\flat Y\to Y$?

The answer, in the strongly semi-abelian context, involves the restriction to the kernel $K$: likewise in the case of groups,
 $\xi$ induces an action $q_*(\xi)$ of the quotient $Q$, precisely when $\xi\cdot(k\flat1)$ is trivial.

\begin{Proposition}\label{prop:act_of_quotients}
Let $\cC$ be strongly semi-abelian, $(k,q)$ a short exact sequence and $\xi$ an action, as above. Then
the following conditions are equivalent:
\begin{itemize}
\item[$(1)$] $\xi\cdot(k\flat1)=\rho_K$, i.e.\ the trivial action on $K$,
\item[$(2)$] there exists an action $q_*(\xi)\colon Q\flat Y\to Y$ such that $\xi=q_*(\xi)\cdot(q\flat1)$.
\end{itemize}
\end{Proposition}
\begin{proof}
This is nothing but the formulation of Proposition \ref{prop:trivial_pointed_pass_to_Q} in terms of internal actions.
\end{proof}

\section*{Further developments}
Theorem \ref{thm}, together with Proposition \ref{prop:act_of_quotients}, seems to suggest that strongly semi-abelian
categories are a convenient setting for developing homological algebra of internal (pre)crossed modules.

Let us recall that strongly semi-abelian varieties include several classical categories of algebras. Indeed, as proved in
\cite{MM10b}, all the  distributive $\Omega_2$-groups, also called categories of  \emph{groups with operations}, are such.
Here we recall the definition for the reader's convenience.

A distributive $\Omega_2$-group is a variety of groups (in the sense of universal algebra) such that:
$\Omega=\Omega_0\bigcup\Omega_1\bigcup\Omega_2$, with
$\Omega_0=\{0\}$, $\Omega_1=\{-\}\bigcup\Omega_1'$ and $\Omega_2=\{+\}\bigcup \Omega_2'$.
where we adopted the additive notation for the (non necessarily commutative) group structure.
These data must satisfy the following axioms
$$
a*(b+c)=a*b+a*c, \qquad (b+c)*a=b*a+c*a,\qquad \mathrm{for\ all\ }*\in \Omega_2';
$$
$$
\omega(a+b)=\omega(a)+\omega(b), \qquad \mathrm{for\ all\ }\omega\in \Omega_1';
$$
$$
\omega(a)*b=\omega(a*b)=a*\omega(b),\qquad \mathrm{for\ all\ }*\in \Omega_2'\mathrm{\ and\ }\omega\in \Omega_1'.
$$
Examples of categories of distributive $\Omega_2$-groups are the categories of groups, rings, Lie algebras, Leibnitz
algebras among others.

Moreover, for all distributive $\Omega_2$-groups, it is possible to  translate conditions involving internal
actions, in conditions involving \emph{external} actions as defined in \cite{Pao}, thus making the theory manageable
in many  algebraic situation of interest.

\section*{Acknowledgements}

I wish to thank Alan Cigoli and Sandra Mantovani,  for their advices throughout the preparation of this paper,
 Zurab Janelidze, Marino Gran and Tim Van der Linden for a useful discussion, the anonymous referee for her/his
 suggestions that have enhanced some aspects of the paper in this final form.

\end{document}